\documentclass[12pt,a4paper,reqno]{amsart}

\usepackage[utf8]{inputenc}
\usepackage[T1]{fontenc}
\usepackage[american]{babel}
\usepackage{amsmath,amssymb,amsthm,amscd}
\usepackage{mathrsfs}
\usepackage[active]{srcltx}
\usepackage{pgf,tikz}
\usepackage{mathrsfs}
\usepackage{enumerate}
\usetikzlibrary{arrows}
\usepackage{paralist}

\usepackage[colorlinks,linkcolor={black},citecolor={black},urlcolor={black}]{hyperref}
\usepackage[margin=2.6cm,bmargin=3cm,tmargin=3cm]{geometry}

\newcommand{\supp}{\operatorname{supp}}

\newcommand{\A}{\mathcal{A}}
\newcommand{\cD}{\mathcal{D}} 
\renewcommand{\P}{\mathcal{P}}
\newcommand{\W}{\mathcal{W}}
\newcommand{\X}{\mathcal{X}}
\newcommand{\Y}{\mathcal{Y}}

\newcommand{\cP}{\mathcal{P}}
\newcommand{\cL}{\mathcal{L}}
\newcommand{\cX}{\mathcal{X}}
\newcommand{\cR}{\mathcal{R}}
\newcommand{\cE}{\mathcal{E}}
\newcommand{\cH}{\mathcal{H}}
\newcommand{\cA}{\mathcal{A}}

\newcommand{\cY}{\mathcal{Y}}
\newcommand{\cW}{\mathcal{W}}
\newcommand{\sd}{\mathsf{d}}
\newcommand{\sD}{\mathsf{D}}

\newcommand{\PX}{\P(\X)}
\newcommand{\PY}{\P(\Y)}
\newcommand{\CE}{\mathsf{CE}}
\newcommand{\HJ}{\mathsf{HJ}}

\newcommand{\hmu}{\widehat\mu}

\newcommand{\R}{\mathbb{R}}
\newcommand{\Z}{\mathbb{Z}}
\newcommand{\N}{\mathbb{N}}
\newcommand{\dd}{\; \mathrm{d}}
\newcommand{\ddd}{\mathrm{d}}

\renewcommand{\a}{\alpha}
\newcommand{\eps}{\varepsilon}

\newcommand{\ip}[1]{\langle {#1}\rangle}
\newcommand{\iip}[1]{\langle\hspace{-0.4ex}\langle {#1}\rangle\hspace{-0.4ex}\rangle}

\newcommand{\bip}[1]{\big\langle {#1}\big\rangle}
\newcommand{\norm}[1]{\| {#1}\|}
\newcommand{\abs}[1]{\lvert#1\rvert}

\newcommand{\ddt}{\frac{\mathrm{d}}{\mathrm{d}t}}

\theoremstyle{plain}
\newtheorem{theorem}{Theorem}[section]

\newtheorem{lemma}[theorem]{Lemma}
\newtheorem{proposition}[theorem]{Proposition}

\newtheorem{definition}[theorem]{Definition}
\newtheorem{assumption}[theorem]{Assumption}

\theoremstyle{remark}
\newtheorem{remark}[theorem]{Remark}
\newtheorem{example}[theorem]{Example}

\newtheorem*{claim*}{Claim}
\newtheorem*{remark*}{Remark}
\newtheorem*{example*}{Example}
\newtheorem*{notation*}{Notation}
\numberwithin{equation}{section}

\definecolor{jan}{rgb}{0.0,0.3,0.8}
\definecolor{mat}{rgb}{0.0,0.5,0.3}

\title{On the geometry of geodesics in discrete optimal transport}
\author{Matthias Erbar}
\address{Institut f\"ur angewandte Mathematik, Universit\"at Bonn, Endenicher Allee 60, 53115 Bonn, Germany}
\email{erbar@iam.uni-bonn.de}
\author{Jan Maas}
\address{Institute of Science and Technology Austria (IST Austria),
Am Campus 1, 3400 Klosterneuburg, Austria}
\email{jan.maas@ist.ac.at}

\author{Melchior Wirth}
\address{Institute of Mathematics, Friedrich Schiller University Jena, 07737 Jena, Germany}
\email{melchior.wirth@uni-jena.de}

\begin{document}

\begin{abstract}
We consider the space of probability measures on a discrete set $\X$, endowed with a dynamical optimal transport metric.
Given two probability measures supported in a subset $\Y \subseteq \X$, it is natural to ask whether they can be connected by a constant speed geodesic with support in $\Y$ at all times. 
Our main result answers this question affirmatively, under a suitable geometric condition on $\Y$ introduced in this paper. 
The proof relies on an extension result for subsolutions to discrete Hamilton--Jacobi equations, which is of independent interest.
\end{abstract}

\maketitle

\tableofcontents

\section{Introduction}
 
Optimal transport continues to be a very active field of research, both in mathematics and in applications. One of the central objects is the $L^p$-Kantorovich metric $W_p$, defined by
\begin{align*}
	W_p(\mu,\nu) = \left(\inf_{\pi \in \Pi(\mu,\nu)}\int_{\cX \times \cX}d(x,y)^p\,d\pi(x,y)\right)^{1/p},
\end{align*}
where $\mu$ and $\nu$ are Borel probability measures on the metric space $(\cX,d)$, and $\Pi(\mu,\nu)$ is the set of all couplings of $\mu$ and $\nu$.

The metric $W_2$ plays a special role in the theory, as it is the crucial object in the gradient flow formulation of dissipative PDE (starting from \cite{JKO98,Otto01}) and in the synthetic theory of Ricci curvature \cite{LoVi09,Stur06}, which builds on McCann's discovery that several important functionals enjoy convexity properties along $W_2$-geodesics \cite{McC97}.

In spite of the robustness of the optimal transport theory, it is well known that if the underlying space is discrete, $W_2$ has several undesirable properties that hamper its usefulness. In particular, if $\cX$ is discrete, the metric space $(\cP(\cX), W_2)$ does not contain any non-trivial geodesics. 

To circumvent this problem, several authors introduced discrete dynamical transport metrics $\cW$, based on discrete versions of the Benamou--Brenier formulation of optimal transport \cite{CHLZ12,Ma11,Mie11}. These metrics have been intensively studied in recent years; in particular, gradient flow formulations have been obtained for nonlinear evolution equations \cite{ErMa14,MaMa16}, and a discrete theory of Ricci curvature has been developed based on geodesic convexity of entropy functionals along discrete optimal transport \cite{ErMa12,Mie13}. Such Ricci curvature bounds have subsequently been obtained in various discrete probabilistic models \cite{ErMaTe15,FaMa16,EHMT17}. 

\medskip

In spite of the relevance of the notion of geodesic convexity, geometric properties of $\cW$-geodesics are currently poorly understood.
The aim of this paper is to obtain results of this type. 
We focus on the issue of \emph{locality} of geodesics in the space of probability measures. 

More precisely, let $(\cX, \sd)$ be a metric space, and consider a geodesic  metric $\sD$ on (a subset of) the space of Borel probability measures $\cP(\cX)$. We say that a subset $\cY \subseteq \cX$ has the \emph{weak locality} property if any pair of probability measures $\mu_0, \mu_1 \in \cP(\cX)$ supported in $\cY$ can be connected by a geodesic that is supported in $\cY$ at  all times. The notion of \emph{strong locality} is defined by requiring this property to hold for \emph{any} geodesic connecting $\mu_0$ and $\mu_1$. If any pair of measures can be connected by a unique geodesic, the notions of weak and strong locality coincide, but this property is currently unknown for discrete dynamical transport metrics.

If $(\cX, \sd)$ is a geodesic metric space, and $\sD$ is the Kantorovich metric $W_p$ for some $1 \leq p < \infty$, it is well known that a subset $\cY$ has the weak (resp. strong) locality property if and only if it is weakly (resp. strongly) geodesically convex. This follows from the fact that geodesics in $(\cP_p(\cX), W_p)$ are supported on geodesics in $(\cX, \sd)$; cf. \cite{Lisi07} for a precise formulation of this result in a general setting.

\medskip

Interestingly, the issue of locality in the discrete setting (with a discrete dynamical transport metric $\cW$ on $\cP(\cX)$ instead of $W_p$) turns out to be much more delicate. 
For example, if one considers the complete graph on a three-point set $K_3$, then any geodesic connecting two Dirac masses transports a nontrivial part of the mass via the third point. Hence, two-point subsets of $K_3$ do not have the locality property. This is shown in Section \ref{sec:triangle} of this paper.

Based on this observation one may conjecture that any nontrivial geodesic  has support on the whole graph. However, we show that this is not the case. In fact, the main contribution of this paper is the introduction of a geometric condition for subsets $\cY \subseteq \cX$ (the \emph{retraction property}), that is shown to be sufficient for locality; see Theorem \ref{thm:geodesics-extension}. The retraction property is easy to check in concrete examples, as is shown in Section \ref{sec:local}. 

As an application of our main result, we show that if $\cX$ is any subset of the grid $\Z^d$ with the usual graph structure, and $\cY \subseteq \cX$ is a hyperrectangle, then any pair of measures supported in $\cY$ can be connected by a geodesic supported in $\cY$. In particular, this property holds for measures supported on subsets of lines, or $k$-dimensional hyperplanes of dimension less than $d$.

A key ingredient in the proof of our main result is a duality result for the discrete transport metric $\cW$, which was recently obtained by Gangbo, Li, and Mou (under slightly more restrictive conditions on the transition rates) \cite{GLM17}. 
We interpret this result (Theorem \ref{thm:dual-W} below) in terms of subsolutions of a discrete Hamilton--Jacobi equation and present a different proof based on Fenchel--Rockafellar duality. 
We then show  that subsolutions of the Hamilton--Jacobi equation on a subset $\cY \subseteq \cX$ can be extended to the full space $\cX$, provided that $\cY$ has the retraction property; cf. Theorem \ref{thm:HJ-extension}. Our main theorem is then a straightforward consequence of this result.

\subsection*{Structure of the paper}
In Section \ref{sec:notation} we collect the necessary preliminaries on discrete transport metrics. Section \ref{sec:dual} contains the dual formulation of the transport problem in terms of Hamilton--Jacobi subsolutions. 
In Section \ref{sec:local} we introduce the retraction property, we show the extension result for subsolutions to the Hamilton--Jacobi equation (Theorem \ref{thm:HJ-extension}), and we prove the main result on weak locality of subsets with the retraction property (Theorem \ref{thm:geodesics-extension}). 
In Section \ref{sec:ends} we show that the strong locality property holds for Markov chains with ``dead ends''.
Finally, it is shown in Section \ref{sec:triangle} that geodesics between Dirac measures on the triangle have full support.

\small
\subsection*{Acknowledgements}

M.E. gratefully acknowledges support by the German Research Foundation through the Hausdorff Centre for Mathematics and the Collaborative Research Centre 1060, \emph{The Mathematics of Emergent Effects}. J.M.gratefully acknowledges support by the European Research Council (ERC) under the European Union's Horizon 2020 research and innovation programme (grant agreement No 716117), and by the Austrian Science Fund (FWF), Project SFB F65. M.W. gratefully acknowledges financial support by the German Academic Scholarship Foundation (Studienstiftung des deutschen Volkes) and by the German Research Foundation (DFG) via RTG 1523.

\normalsize

\section{The discrete transport distance}
\label{sec:notation}

In this section we briefly recall the definition and basic properties of the discrete transport distance constructed in \cite{CHLZ12,Ma11,Mie11}.

Let $\X$ be a finite set, and let $Q: \X \times \X \to \R_+$ denote the 
transition rates for a Markov chain on $\cX$. Without loss of generality, we use the convention that $Q(x,x) = 0$ for all $x \in \cX$.
The corresponding generator $\cL$ acts on functions $\phi : \cX \to \R$ by
\begin{align*}
  \cL \phi(x) = \sum_{y \in \cX} Q(x,y)\big(\phi(y)-\phi(x)\big)\;.
\end{align*}
We assume that $Q$ is \emph{irreducible}, i.e., each pair $(x, y) \subseteq \cX \times \cX$ can be connected, for some $n \in \N$, by a path $\{x_i\}_{i=0}^n$ satisfying $x_0 = x$, $x_n = y$, and $Q(x_{i-1}, x_i) > 0$ for $i = 1, \ldots, n$. 
This assumption implies the existence of a unique stationary probability measure $\pi$ on $\X$. Moreover, $\pi$ is strictly positive. 
We will furthermore assume that $Q$ is \emph{reversible} with respect to $\pi$, i.e., the \emph{detailed balance condition} holds:
\begin{align*}
	 \pi(x) Q(x,y) = \pi(y) Q(y,x) \quad \text{ for all } x,y \in \cX \;.
\end{align*}
The triple $(\X,Q,\pi)$ will be referred to as a \emph{Markov triple}. 

A Markov chain induces a graph on the vertex set $\cX$, whose edge set   is given by $\cE = \{ (x,y) \in \cX \times \cX : Q(x,y) > 0 \}$. 
We write $x \sim y$ iff $Q(x,y) > 0$.
The assumption that $Q$ is irreducible corresponds to the graph $(\cX,\cE)$ being connected. The detailed balance condition implies that the graph is undirected.

\medskip 

In order to define the discrete transport distance on the set $\PX$ of
probability measures on $\X$, we introduce the following objects.

\begin{definition}[Continuity equation]
  A pair $(\mu,V)$ is said to satisfy the \emph{continuity
    equation} if
  \begin{itemize}
  \item[(i)] $\mu:[0,T]\to\R$ is continuous;
  \item[(ii)] $V:[0,T]\to\R^{\X\times\X}$ is locally integrable;
  \item[(iii)] $\mu_t\in\PX$ for all $t\in[0,T]$;
  \item[(iv)] the continuity equation holds in the sense of
    distributions:
    \begin{align}\label{eq:cont-eq}
      \ddt \mu_t(x) + \frac12\sum_{y\in\X}\Big( V_t(x,y) - V_t(y,x) \Big) = 0 \quad \quad \text{ for all } x\in\X\;.
    \end{align}
  \end{itemize}
  In this case, we write $(\mu,V) \in \CE_T$. Furthermore, 
  $\CE_T(\mu^0,\mu^1)$ denotes the collection of pairs $(\mu,V)\in\CE_T$ satisfying $\mu|_{t=0}=\mu^0$ and $\mu|_{t=T}=\mu^1$.
\end{definition}

\begin{definition}[Admissible mean]\label{def:mean}
An \emph{admissible mean} is a continuous function
$\Lambda: \R_+ \times \R_+ \to \R_+$ that is $C^\infty$ on
$(0,\infty) \times (0,\infty)$, symmetric, positively $1$-homogeneous,
non-decreasing in each of its variables, jointly concave, and
normalised, i.e., $\Lambda(1,1) = 1$. 
\end{definition}

Of particular interest to us is the \emph{logarithmic mean} given by
\begin{align*}
  \Lambda_{\text{log}}(s,t):=\int_0^1s^\alpha t^{1-\alpha}\dd \alpha\;,
\end{align*}
since it arises in the entropic gradient flow structure for the master equation $\partial_t \mu = \cL^* \mu$.  
Other relevant examples of admissible means are the harmonic mean $\Lambda_{\text{har}}(s,t) = \frac{2st}{s+t}$, the geometric mean $\Lambda_{\text{geo}}(s,t) = \sqrt{st}$, and the arithmetic mean $\Lambda_{\text{ari}}(s,t) = \frac{s+t}{2}$. Some of these means arise in gradient structures for porous medium equation; cf. \cite{ErMa14}. From now on, we will fix an admissible mean $\Lambda$. 

\medskip

The action functional for the discrete transport distance is defined using  the convex and lower semicontinuous function $A: \R^3 \to [0,\infty]$ given by 
\begin{align}\label{eq:def-A}
A(s,t,w) :=
  \begin{cases}
    \frac{w^2}{\Lambda(s,t)}\;,& w \in \R,\ s,t>0\;,\\
    0\;, & w=0,\ s,t\geq0\;,\\
    +\infty\;, &\text{otherwise} \;.
  \end{cases}
\end{align}
For $\mu \in \PX$ and $V: \cX \times \cX \to \R$ we define the action by
\begin{align*}
  \A(\mu,V) =\frac12\sum_{x,y\in\X}A\big(\mu(x)Q(x,y),\mu(y)Q(y,x), V(x,y)\big)\;.
\end{align*}
For brevity we sometimes write
\begin{align*}
	\hmu(x,y) := \Lambda\big( \mu(x)Q(x,y), \mu(y) Q(y,x) \big)\;. 
\end{align*}

\begin{definition}[Discrete transport distance]
For a Markov triple $(\cX, Q, \pi)$ and an admissible mean $\Lambda$, the \emph{discrete transport distance} $\W$ is defined for
  $\mu_0,\mu_1\in\PX$ by
  \begin{align}\label{eq:def-W}
    \W(\mu_0,\mu_1):=\inf\Bigg\{\sqrt{\int_0^1\A(\mu_t,V_t)\dd t} : (\mu,V)\in\CE_1(\mu_0,\mu_1)\Bigg\}\;.
  \end{align}
\end{definition}

It has been shown in \cite{ErMa12} that minimisers exist in the minimisation problem above. Any minimal curve $(\mu_t)_{t\in[0,1]}$ is a constant speed geodesic, i.e., it satisfies $\W(\mu_s,\mu_t)=|t-s|\W(\mu_0,\mu_1)$ for all $s,t\in[0,1]$. 

\begin{remark}\label{rem:anti-symmetry}
Without loss of generality we may assume in the minimisation \eqref{eq:def-W} that $V$ is anti-symmetric, i.e., $V_t(x,y)= -V_t(y,x)$. In fact, for each $U \in \R$, the quantity $|V(x,y)|^2+|V(y,x)|^2$ is minimised
among all choices of $V(x,y), V(y,x)$ such that $V(x,y) - V(y,x) = U$
by choosing $V(y,x) = -V(x,y) = U/2$.
\end{remark}

Finally, let us introduce some convenient notation to be used in the
sequel. We denote the Euclidean inner products on $\R^\X$ and
$\R^{\X\times\X}$ by
\begin{align*}
 \ip{\phi, \psi} = \sum_{x \in \cX} \phi(x) \psi(x) 
 \quad \text{and} \quad
  \iip{\Phi, \Psi} = \frac12\sum_{x,y \in \cX} \Phi(x,y) \Psi(x,y) \ . 
\end{align*}
The discrete gradient of a function $\phi \in \R^\cX$ will be denoted by
$\nabla \phi(x,y) = \phi(y) - \phi(x)$, and the discrete divergence of $\Phi\in\R^{\X\times\X}$ is given by 
\begin{align*}
  \nabla\cdot\Phi(x) = \frac12\sum_{y\in\X}\big(\Phi(x,y)-\Phi(y,x)\big)\;.
\end{align*}

Furthermore, for $\mu \in \PX$ and $\Phi\in\R^{\X\times\X}$ we write
\begin{align*}
 \| \Phi \|_\mu:= \sqrt{\iip{\Phi,\Phi\cdot\hat \mu}}\;.
\end{align*} 
where the multiplication of $\Phi\cdot\hat \mu$ is understood componentwise.
For all $\Phi,V\in\R^{\X\times\X}$ and $\mu\in\PX$, Young's inequality yields
\begin{align}\label{eq:CS}
  \iip{\Phi,V} \leq \frac12 \|\Phi\|_\mu^2 + \frac12 \cA(\mu,V)\;.
\end{align}

\section{Duality for discrete optimal transport}
\label{sec:dual}

We present a dual formulation for the discrete transport distance which can be seen as a discrete analogue of the Kantorovich duality. This result has recently been proved in \cite{GLM17} using different methods; cf. Proposition 3.10 and Theorems 5.10 and 7.4 in that paper. Note that the result in \cite{GLM17} is stated under slightly stronger assumptions on the transition rates. In our notation, it is assumed there that $Q(x,y) = Q(y,x)$ and $\pi$ is constant. The slightly greater generality here does not cause additional difficulties.

\begin{definition}[Hamilton--Jacobi subsolution]\label{def:HJ}
A function $\phi\in H^1\big((0,T);\R^\cX\big)$ is said to be a  \emph{Hamilton--Jacobi subsolution} if for a.e. $t$ in $(0,T)$, we have
\begin{align}\label{eq:HJ}
 \ip{\dot\phi_t,\mu} 
 	+ \frac12 \|\nabla\phi_t\|^2_\mu 
		\leq 0 
			\quad \text{ for all } \mu \in \cP(\cX)\;.  
\end{align}
The collection of all Ha\-mil\-ton--Jacobi subsolutions is denoted $\HJ^T_\cX$.
\end{definition}

\begin{remark}\label{rem:HJ-scaling}
 Given $\phi \in \HJ_\cX^T$ and $\lambda>0$, set $\phi^\lambda_t:=\lambda \phi_{\lambda t}$. Then $\phi^\lambda\in \HJ_\cX^{\lambda T}$.
\end{remark}

\begin{theorem}[Duality formula]\label{thm:dual-W}
For $\mu_0,\mu_1\in\PX$ we have
\begin{align}\label{eq:dual-W}
 \frac12\cW^2(\mu_0,\mu_1)=\sup\big\lbrace
\ip{\phi_1,\mu_1}-\ip{\phi_0,\mu_0}\,:\, \phi\in \HJ^1_\cX\big\rbrace\;.
\end{align}
  This representation remains true if the supremum is restricted to the class of functions $\phi\in C^1\big([0,1],\R^\cX\big)$ satisfying \eqref{eq:HJ}.

\end{theorem}

Let us first give a heuristic argument for the duality result above. We start by
introducing a Lagrange multiplier for the continuity equation
constraint and write
\begin{align}\label{eq:dual1}
  \frac12\cW^2(\mu_0,\mu_1)&=\inf\limits_{\mu,V}\sup\limits_\phi\left\{\int_0^1\frac12\cA(\mu_t,V_t)\dd
    t + \int_0^1\ip{\phi_t,\dot\mu_t+\nabla\cdot V_t} \dd
    t\right\}\;,
\end{align}
where the supremum is taken over all (sufficiently smooth) functions
$\phi:[0,1]\to\R^\cX$ and the infimum is taken over all (sufficiently
smooth) curves $\mu: [0,1] \to \R_+$ connecting $\mu_0$ and
$\mu_1$, and over all $V : [0,1] \to \R^{\cX \times \cX}$. Here we do
not require that $(\mu,V)$ satisfies the continuity equation, but the
inner supremum takes the value $+\infty$ if $(\mu,V)$ does not belong
to $\CE_1(\mu_0,\mu_1)$. We also do not require that $\mu$ takes
values in $\PX$, but this is automatically enforced by the continuity
equation. 

Integrating by parts and using the min--max principle we obtain
\begin{align*}
  \frac12\cW^2(\mu_0,\mu_1)
  & = \inf\limits_{\mu,V}\sup\limits_\phi\left\{\ip{\phi_1,\mu_1}-\ip{\phi_0,\mu_0}
    + \int_0^1\frac12\A(\mu_t,V_t)
    -\ip{\dot\phi_t,\mu_t}-\iip{\nabla\phi_t, V_t} \dd
    t\right\}
 \\& = \sup\limits_\phi\left\{\ip{\phi_1,\mu_1}-\ip{\phi_0,\mu_0}
    + \inf\limits_{\mu,V} \int_0^1\frac12\A(\mu_t,V_t)
    -\ip{\dot\phi_t,\mu_t}-\iip{\nabla\phi_t, V_t} \dd
    t\right\}.
\end{align*}
As the quantity to be minimised is positively $1$-homogeneous in $(\mu, V)$, the
infimum takes the value $-\infty$ if $\phi$ does not belong to $\cH$, the set of $C^1$ functions $\phi:[0,1]\to\R^{\cX}$ satisfying 
\begin{align*}
  \int_0^1\frac12\A(\mu_t,V_t)
  -\ip{\dot\phi_t,\mu_t}-\iip{\nabla\phi_t, V_t} \dd t \geq 0 
\end{align*}
for all $\mu : [0,1] \to \R_+^\cX$ and all $V : [0,1] \to \R^{\cX
  \times \cX}$. 
Consequently, 
\begin{align*}
  \frac12\cW(\mu_0,\mu_1)^2=\sup\Big\{\ip{\phi_1,\mu_1}-\ip{\phi_0,\mu_0}
    \ : \ \phi\in\cH\Big\}\;.
\end{align*}
A simple localisation argument in $t$ shows that $\phi\in\cH$ iff for all $t\in[0,1]$ and $(\mu,V)\in
\R_+^\cX\times\R^{\cX\times\cX}$:
\begin{align*}
 \frac12\A(\mu,V) -\ip{\dot\phi_t,\mu} - \iip{\nabla\phi_t,V} \geq 0 \;.
\end{align*}
We may write 
\begin{align*}
 \frac12\A(\mu,V) - \iip{\nabla\phi_t,V}
  &= \frac14\sum_{x,y}\frac{\big(V(x,y) - \hmu(x,y) \nabla\phi_t(x,y)\big)^2}{\hmu(x,y)}
 - \frac12\norm{\nabla\phi_t}^2_\mu\;.
\end{align*}
Minimising over $V$ we conclude that $\phi \in \cH$ iff the
inequality
\begin{align*}
\ip{\dot\phi_t,\mu}  + \frac12\norm{\nabla\phi_t}^2_\mu \leq 0
\end{align*}
holds for all $\mu \in \R^\cX_+$ and $t \in [0,1]$, which means that $\phi \in \HJ_\cX$.

\medskip

We present a proof of Theorem \ref{thm:dual-W} using the Fenchel--Rockafellar duality theorem; see, e.g., \cite[Theorem 1.9]{Vil03}. Recall that given a normed vector
space $E$ with topological dual space $E^*$ and a proper convex function
$F: E \to \R \cup \{+\infty\}$, its Legendre--Fenchel transform
$F^* : E^* \to \R \cup \{+\infty\}$ is defined by
\begin{align*}
 F^* :E^*\to \R\cup\{+\infty\}\;, \qquad 
 F^*(x^*) := \sup\limits_{x \in E}\Big\{ \ip{x,x^*}  - F(x) \Big
\}\;.
\end{align*}

\begin{theorem}[Fenchel--Rockafellar duality]\label{thm:FR-duality}
  Let $E$ be a normed vector space and $E^*$ its topological dual. Let
  $F,G:E\to\R\cup\{+\infty\}$ be proper convex functions and denote by
  $F^*,G^*:E^*\to \R\cup\{+\infty\}$ their Legendre--Fenchel
  transforms. Assume that there is $z_0\in E$ such that $G$ is
  continuous at $z_0$ and $F(z_0),G(z_0)<\infty$. Then we have:
  \begin{align}\label{eq:FR-duality}
    \sup\limits_{z\in E}\Big[ - F(z) - G(z) \Big] = \min\limits_{z^*\in E^*} \Big[ F^*(z^*) + G^*(-z^*) \Big]\;.
  \end{align}
\end{theorem}

\begin{proof}[Proof of Theorem \ref{thm:dual-W}]
Let us first note that, by the convexity of the constraint \eqref{eq:HJ}, any $\phi\in \HJ^1_\cX$ can be approximated uniformly by $C^1$ functions satisfying \eqref{eq:HJ} by convolution (after scaling the function to a slightly larger interval $[-\delta,1+\delta]$ via Remark \ref{rem:HJ-scaling}). Therefore, the final part of the theorem follows.

To show the dual representation with $C^1$ functions, we will apply Theorem \ref{thm:FR-duality} in the following situation. 
Let $E$ be the Banach space 
  \begin{align*}
    E = C^1\big([0,1],\R^\cX\big) \times L^2\big((0,1),\R^{\cX\times\cX}\big) \;.
  \end{align*}
Since we can identify $C^1\big([0,1],\R^{\cX}\big)$ with $\R^\cX \times C^0\big([0,1],\R^{\cX}\big)$ via the map
$I : \phi \mapsto (\phi_0, \dot\phi)$,
the dual space $E^*$ can be identified with
  \begin{align*}
    E^*= \R^\cX \times M\big([0,1],\R^\cX\big)  \times L^2\big((0,1),\R^{\cX\times\cX}\big)\;,
  \end{align*}
where the duality pairing between $(\phi, \Phi)\in E$ and $(b,\sigma,V)\in E^*$ is given by
  \begin{align*}
    \ip{(\phi,\Psi), (b,\sigma,V)} &=
    \ip{\phi_0,b} 
     + \int_0^1\ip{\dot\phi_t,\dd\sigma(t)}
     + \int_0^1\iip{\Phi_t,V_t}\dd t\;,
  \end{align*}
  keeping in mind that $\sigma$ is a vector-valued measure.

Define the functionals $F,G: E \to \R \cup\{+\infty\}$ by
  \begin{align*}
    F(\phi,\Phi) &=
    \begin{cases}
      -\ip{\phi_1,\mu_1} + \ip{\phi_0,\mu_0}\;, & \Phi = \nabla\phi\;,\\
       +\infty\;, & \text{otherwise}\;,
    \end{cases}\\
   G(\phi,\Phi) &=
    \begin{cases}
      0\;, & (\phi,\Phi)\in \cD\;,\\
       +\infty\;, & \text{otherwise}\;.
    \end{cases}   
  \end{align*}
  Here we say that a pair $(\phi,\Phi)\in E$ belongs to $\cD$ if for
  all~$t\in(0,1)$ and all $\mu\in\PX$ we have
  $\ip{\dot\phi_t,\mu}+\frac12 \|\Phi_t\|_\mu^2 \leq 0$. It is readily
  checked that $F$ and $G$ are convex.  Moreover, setting
  $\phi_0(t) = t(-1,\dots,-1)$ and $\Phi_0 \equiv 0$, both $F$ and $G$
  are finite at $(\phi_0,\Phi_0)$ and $G$ is continuous. Note that the
  supremum in the left-hand side of \eqref{eq:FR-duality} coincides
  with the supremum in the right-hand side of \eqref{eq:dual-W}.

We will calculate the Legendre--Fenchel transforms of $F$ and $G$. For $F$ we obtain
 \begin{align}\nonumber
 F^*(b,\sigma,V)&=
   \sup\limits_{(\phi,\Phi)\in E} 
   \left\{ \bip{(\phi,\Phi), (b,\sigma,V)} 
   	- F(\phi,\Phi) 
	\right\}\\\nonumber
   &= \sup\limits_{\phi} \left\{
      \ip{\phi_0, b} 
      + \int_0^1 \ip{\dot\phi_t,\ddd\sigma(t)} +
     \int_0^1\iip{\nabla\phi_t,V_t}\dd t +\ip{\phi_1,\mu_1} - \ip{\phi_0,\mu_0} 
     \right\}.
 \end{align}
Thus, by homogeneity of the last expression in $\phi$, one has $F^*(b,\sigma,\nu)= + \infty$ unless $(\sigma, V)$ satisfies the continuity equation $\partial_t \sigma + \nabla \cdot V = 0$ with boundary values $-(\mu_0-b)$ and $-\mu_1$, in the sense that
 \begin{align}\label{eq:weaker-ce}
   \ip{\phi_1,-\mu_1} - \ip{\phi_0,-(\mu_0-b)}
   =\int_0^1\ip{\dot\phi_t,\dd\sigma(t)} 
     + \int_0^1\iip{\nabla\phi_t,V_t}\dd t
 \end{align}
for all $\phi\in C^1\big([0,1],\R^\cX\big)$. In particular, the distributional derivative of $\sigma$ belongs to $L^2([0,1];\R^\cX)$. Since the antiderivative of a distribution is unique up to a constant,  the fundamental theorem of Lebesgue calculus implies that $\sigma$ has the form $\mathrm{d}\sigma(t)=\sigma_t\dd t$ for some curve $(\sigma_t)_t\in H^1([0,1];\R^\X)$. Moreover, (\ref{eq:weaker-ce}) implies $\sigma_0=-(\mu_0-b)$ and $\sigma_1=-\mu_1$.
 Thus, we obtain
 \begin{align}  \label{eq:F-dual}
  F^*(b,\sigma,V)&=
  \begin{cases}
    0\;, & (-\sigma,-V)\in\CE'(\mu_0-b,\mu_1)\;,\\
    +\infty\;, & \text{otherwise}\;,
  \end{cases}
 \end{align}
where $\CE'$ is defined by dropping the positivity and normalisation condition $(i)$ in the definition of $\CE$, and we have identified the measure $\sigma$ with the $H^1$-map $\sigma_t$. 

As it suffices to calculate the transform of $G$ at points $(\sigma,b,V)$ where $F^*$ is finite,  we can assume that $\ddd \sigma(t)= \sigma_t\dd t$ with $(\sigma_t)_{t}\in H^1([0,1];\R^\X)$. We claim that:
 \begin{align} \label{eq:G-dual}
  G^*(b,\sigma,V)
  =
  \begin{cases}
    \frac12\int_0^1\cA(\sigma_t,V_t)\dd t\;, & b=0\;,\\
    +\infty\;, & \text{otherwise}\;.
  \end{cases}
 \end{align}                
 Indeed, it follows that
 \begin{align*}
   G^*(b,\sigma,V)&= 
   \sup\limits_{(\phi,\Phi)\in E} 
   \left\{ \bip{(\phi,\Phi), (b,\sigma,V)}
		 	- G(\phi,\Phi) \right\}
			\\
   &=\sup\limits_{(\phi,\Phi)\in\cD} \left\{ \ip{\phi_0, b} 
   		+ \int_0^1\ip{\dot\phi_t, \sigma_t} 
		+ \iip{\Phi_t,V_t}\dd t \right\}\;.
  \end{align*}
Since $(\phi, \Phi) \in \cD$ implies $(\phi + c, \Phi) \in \cD$ for all $c \in \R^\cX$, we have $G^*(b,\sigma,V)=+\infty$ unless $b=0$. Moreover, from the definition of $\cD$ we infer that $G^*(b,\sigma,V)=+\infty$ unless $\sigma_t\in\R_{+}^\cX$ for a.e.~$t$.

Let us assume that $b=0$ and $\int_0^1\cA(\sigma_t,V_t)\dd t<\infty$. Then we obtain
\begin{equation}\begin{aligned}
\label{eq:G-dual2}
   G^*(0,\sigma,V) 
      &=\sup\limits_{(\phi,\Phi)\in\cD} 
      \left\{ \int_0^1\ip{\dot\phi,\sigma_t} + \iip{\Phi_t,V_t} \dd t \right\}\\
    &\leq\sup\limits_{(\phi,\Phi)\in\cD} \bigg\{ \int_0^1  - \frac12\|\Phi_t\|_{\sigma_t}^2 + 
\iip{\Phi_t,V_t}  \dd t \bigg\}  
 \leq  \frac12 \int_0^1 \cA(\sigma_t, V_t) \dd t
\;, 
\end{aligned}\end{equation}
where the first inequality follows from the definition of $\cD$ and
the second from \eqref{eq:CS}.
  
It remains to show that we have in fact equality. First we consider a
convolution in time yielding smooth pairs $\sigma^\eps_t$, $V^\eps_t$
converging to $\sigma_t$, $V_t$ as $\eps \to 0$.  Then we set for
$\delta > 0$, $\sigma^{\delta,\eps}_t = \sigma_t^\eps + \delta\pi$.
By convexity of the action and monotonicity of the mean $\Lambda$ we
have
   \begin{align}\label{eq:majorant}
    \int_0^1\cA(\sigma^{\delta,\eps}_t,V^{\eps}_t)\dd t 
    \leq \int_0^1\cA(\sigma^{\eps}_t,V^{\eps}_t)\dd t
   \leq
    \int_0^1\cA(\sigma_t,V_t)\dd t\;.
   \end{align}
   The convexity and lower semicontinuity of $A$ further implies the lower semicontinuity of the action; see \cite[Theorem 3.4.3]{But89} for a general result on lower semicontinuity of integral functionals and the proof of \cite[Theorem 3.2]{ErMa12} for the application to the action functional $\cA$. Consequently, 
\begin{align*}
 \int_0^1\cA(\sigma_t,V_t)\dd t=
   \lim_{\delta\to0} \lim_{\eps\to0} \int_0^1\cA(\sigma^{\delta,\eps}_t,V^{\eps}_t)\dd t \;.
\end{align*} 
Now, we can choose in particular
  $(\phi^{\delta,\eps},\Phi^{\delta,\eps})$ such that
  \begin{align*}
    \Phi^{\delta,\eps}_t = \frac{V^{\eps}_t}{\hat \sigma^{\delta,\eps}_t}\;,\quad \dot\phi^{\delta,\eps}(x)=-\frac12\sum_y\partial_1\Lambda\big(\rho^{\delta,\eps}_t(x),\rho^{\delta,\eps}_t(y)\big)\abs{\Phi^{\delta,\eps}_t(x,y)}^2Q(x,y)\;,
  \end{align*}
  where $\sigma^{\delta,\eps}=\rho^{\delta,\eps}\pi$. 

  We claim that $(\phi^{\delta,\eps},\Phi^{\delta,\eps})\in \cD$. To
  see this, we use the inequality
  \begin{align*}
  \partial_1\Lambda(s,t)u+\partial_2\Lambda(s,t)v\geq
  \Lambda(u,v) \quad\forall s,t>0\;,\ u,v\geq0\;,
  \end{align*}
  which is an identity for $s=v, t=u$, see \cite[Lemma 2.2]{ErMa12}.
  From this we infer that for any $\mu=\tilde\rho\pi\in\PX$ we have
\begin{equation}\begin{aligned}\label{eq:HJ-eps-delta}
    \ip{\dot\phi_t^{\delta,\eps},\mu} 
    &= -\frac12\sum_{x,y}\partial_1\Lambda\big(\rho_t^{\delta,\eps}(x),\rho_t^{\delta,\eps}(y)\big)\tilde\rho(x)\abs{\Phi^{\delta,\eps}_t(x,y)}^2Q(x,y)\pi(x) \\
  &= -\frac14\sum_{x,y}\Big[\partial_1\Lambda\big(\rho^{\delta,\eps}(x),\rho^{\delta,\eps}(y)\big)\tilde\rho(x)+\partial_2\Lambda\big(\rho^{\delta,\eps}(x),\rho^{\delta,\eps}(y)\big)\tilde\rho(y)\Big]
  \\& \qquad\qquad \times\abs{\Phi_t^{\delta,\eps}(x,y)}^2 Q(x,y)\pi(x)\\
&\leq -\frac14\sum_{x,y}\Lambda\big(\tilde\rho(x),\tilde\rho(y)\big)\abs{\Phi^{\delta,\eps}_t(x,y)}^2 Q(x,y)\pi(x) = -\frac12\| \Phi^{\delta,\eps}_t\|_\mu^2\;,
\end{aligned}\end{equation}
which proves the claim. 
Note that for $\tilde\rho=\rho^{\delta,\eps}_t$ we obtain equality.

Next we claim that 
  \begin{align*}
    &\lim_{\delta\to0} \lim_{\eps\to0}\int_0^1\ip{\dot\phi^{\delta,\eps},\sigma_t}\dd t=-\frac12\int_0^1\cA(\sigma_t,V_t)\dd t\;.
  \end{align*}
To prove this, we compare the left-hand side and the second line in \eqref{eq:HJ-eps-delta}.
The limit $\eps \to 0$ is justified by dominated convergence, since \eqref{eq:HJ-eps-delta} yields the majorant
\begin{align*}
	\frac12 \| \Phi^{\delta,\eps}_t\|_{\sigma_t}^2
	\leq \frac{C}{\delta} \cA(\sigma_t^\eps, V_t^\eps) \;,
\end{align*}
where $C$ depends on $Q$ and $\pi$. The right-hand side converges as $\eps \to 0$ by \eqref{eq:majorant}. 
The limit $\delta\to0$ is justified by monotone convergence.
Similarly, we have
  \begin{align*}
    \lim_{\delta\to0} \lim_{\eps\to0} \int_0^1\iip{V_t, \Phi_t^{\eps,\delta}} \dd t 
    = \int_0^1\cA(\sigma_t,V_t)\dd t\;.
  \end{align*}
  Here, we can use the estimate $|ab|\leq\frac12a^2+\frac12b^2$ to
  obtain a majorant that converges by \eqref{eq:majorant} as
  before. Thus the expression in the first braced bracket of
  \eqref{eq:G-dual2} converges to the right-hand side of
  \eqref{eq:G-dual2} with this choice of
  $(\phi^{\delta,\eps},\Psi^{\delta,\eps})$ as $\delta,\eps\to0$. A
  similar argument yields $G^*(0,\sigma,V)=\infty$ if
  $\int_0^1\cA(\sigma_t,V_t)\dd t=\infty$.
Combining \eqref{eq:F-dual}, \eqref{eq:G-dual} and the fact that $\cA(\sigma,V)=+\infty$ unless $\sigma\in\R^\cX_+$, we obtain
 \begin{align*}
   F^*(-b,-\sigma,-V)+ G^*(b,\sigma,V) =
   \begin{cases}
     \frac12\int_0^1\cA(\sigma_t,V_t)\dd t\;, & (\sigma,V)\in\CE(\mu_0,\mu_1),\ b=0\;,\\
     +\infty\;, & \text{otherwise}\;.
   \end{cases}
 \end{align*}
 Thus the supremum in the right-hand side of \eqref{eq:FR-duality}
 coincides with $\frac12\cW(\mu_0,\mu_1)^2$. An application of Theorem
 \ref{thm:FR-duality} concludes the proof.
\end{proof}

\section{Locality of optimal curves}
\label{sec:local}

In this section we investigate locality properties for discrete
transport geodesics. More precisely, we study the following
question: Given two probability measures supported in a subset $\cY$ of a state space $\cX$, is there an optimal curve connecting them that is supported in $\cY$? The crucial tool to analyse this question is the dual characterisation of the transport problem given in the previous section. We prove two types of results. 

Firstly, we show that the question can be answered affirmatively,
under a simple condition (the \emph{retraction property} of the
subgraph $\cY$), which will be introduced below.  This property
ensures that any competitor in the dual problem on the subgraph can be
extended to the full graph.  We present several examples where this
property is satisfied.  
Later, in Section \ref{sec:triangle}, we will show that locality may fail if the retraction property is not
satisfied.

\medskip

We start by introducing the retraction property and we give several
examples.
To increase readability, we often write subscripts instead of parentheses, e.g., $Q_{xy} = Q(x,y)$.

\medskip

A subset $\cY \subseteq \cX$ is said to be \emph{connected} if any two distinct points $y, y' \in \cY$ can be connected by a path $\{y_i\}_{i=0}^n \subseteq \cY$ satisfying $y_0 = y$, $y_n = y'$, and $Q(y_{i-1}, y_i) > 0$ for $i=1,\ldots,n$.

\begin{definition}[Retraction property]\label{def:extension}
A  connected subset $\Y \subseteq \X$ has the \emph{retraction property} if there exists a map $T\colon \X\to \Y$ such that
\begin{enumerate}
\item[(R1)] $T(y)=y$ for all $y\in \Y$;
\item[(R2)] For all $y,y' \in \cY$ with $y\neq y'$, and all $x \in T^{-1}(y)$, we have
  \begin{align*}
  \sum_{x'\in T^{-1}(y')} Q(x,x') \leq Q(y,y')\;.
  \end{align*}
\end{enumerate}
The map $T$ is called a \emph{retraction} of $\X$ onto $\Y$.
\end{definition}

\begin{remark}\label{rem:simple-retraction}
If the Markov triple $(\X,Q,\pi)$ corresponds to a simple random walk (i.e., $Q(x, y) \in \{0,1\}$ for all $x, y \in \cX$), the retraction property can be rephrased in graph theoretical terms. Indeed, it is readily verified that the retraction property holds if and only if there exists a map $T : \cX \to \cY$ with the following properties:
\begin{enumerate}
\item[\emph{(R1}$'$\emph{)}] $T(y) = y$ for all $y \in \cY$;
\item[\emph{(R2}$'$\emph{)}] If $x \sim x'$, then $T(x) = T(x')$ or $T(x) \sim T(x')$;
\item[\emph{(R3}$'$\emph{)}] If $x_1' \sim x$, $x_2' \sim x$, and $T(x_1') = T(x_2')$ for some $x_1' \neq x_2'$, then $T(x) = T(x_1')$.
\end{enumerate}
\end{remark}

\begin{definition}[Restriction]\label{def:restriction}
The \emph{restriction} of a Markov triple $(\X,Q,\pi)$ to a connected subset $\cY \subseteq \cX$ is the Markov triple $(\cY,Q|_{\cY},\pi|_{\cY})$, where $Q|_{\cY}$ is the restriction of $Q$ to $\cY \times \cY$, and $\pi|_{\cY}$ is the normalised restriction of $\pi$ to $\cY$. 
\end{definition}

Connectedness of $\Y$ implies that the Markov triple $(\cY,Q|_{\cY},\pi|_{\cY})$ is irreducible, and the detailed balance relation is obviously inherited. The following result implies that if $\cY$ has the retraction property as a subset of $\cX$, it also has this property as a subset of any set $\cX'$ with $\cY \subseteq \cX' \subseteq \cX$.

\begin{lemma}\label{lem:domain-monotonicity}
Let $(\cX, Q, \pi)$ be a Markov triple and $\cY \subseteq \cX' \subseteq \cX$. If $T : \cX \to \cY$ is a retraction, then its restriction $T|_{\cX'} : \cX' \to \cY$ is a retraction as well.
\end{lemma}

\begin{proof}
This follows immediately from the definition.
\end{proof}

We present some examples of sets with the retraction property.

\begin{example}[Cycle]\label{ex:cycle}
For $n \geq 2$, let $\X=\Z/n\Z$, and set $Q_{j,j+1}=Q_{j+1,j}=1$ and $Q_{ij}=0$ otherwise. All computations are to be understood modulo $n$. We claim that the subset $\{1,\dots,k\}$ of $\cX$ has the retraction property if and only if $2k \leq n$. In this case, a retraction is given as follows (cf. Figure \ref{fig:nonagon}):
\begin{align*}
T : \cX \to \{1,\dots,k\}\;,\qquad 
	T(j) = 
	\begin{cases}j 
	       & \text{if } 1    \leq j \leq k \;,\\
 	2k-j+1 & \text{if } k+1  \leq j \leq 2k\;,\\
		 1 & \text{if } 2k+1 \leq j \leq n \;.
	\end{cases}
\end{align*}
\end{example}
Indeed, to check sufficiency, note that $(R1')$ is trivial, $(R2')$ holds since $n \geq 2k$, and $(R3')$ is readily checked as well. Necessity follows from a simple argument. 

\definecolor{qqqqff}{rgb}{0.,0.,1.}
\definecolor{ffqqqq}{rgb}{1.,0.,0.}
\definecolor{uuuuuu}{rgb}{0.26666666666666666,0.26666666666666666,0.26666666666666666}
\definecolor{qqffqq}{rgb}{0.,1.,0.}
\definecolor{ffffqq}{rgb}{1.,1.,0.}
\definecolor{ffffff}{rgb}{1.,1.,1.}
\begin{figure}[h]
\begin{tikzpicture}
\draw (250:1.5)--(290:1.5)--(330:1.5)--(10:1.5)--(50:1.5)--(90:1.5)--(130:1.5)--(170:1.5)--(210:1.5)--(250:1.5);
\draw[fill=black] (210:1.5) circle(1.5pt);
\draw (210:1.5) node[anchor=north east] {$y_1$};
\draw[fill=black] (250:1.5) circle(1.5pt);
\draw (250:1.5) node[anchor=north] {$y_2$};
\draw[fill=black] (290:1.5) circle(1.5pt);
\draw (290:1.5) node[anchor=north] {$y_3$};
\draw[fill=black] (330:1.5) circle(1.5pt);
\draw (330:1.5) node[anchor=north west] {$y_4$};
\draw[fill=black] (10:1.5) circle(1.5pt);
\draw (10:1.5) node[anchor=west] {$y_4$};
\draw[fill=black] (50:1.5) circle(1.5pt);
\draw (50:1.5) node[anchor=south west] {$y_3$};
\draw[fill=black] (90:1.5) circle(1.5pt);
\draw (90:1.5) node[anchor=south] {$y_2$};
\draw[fill=black] (130:1.5) circle(1.5pt);
\draw (130:1.5) node[anchor=south east] {$y_1$};
\draw[fill=black] (170:1.5) circle(1.5pt);
\draw (170:1.5) node[anchor=east] {$y_1$};
\draw[dashed](-2.5,-0.3) rectangle(2.5,-2);
\draw (2.5,-0.3) node [anchor= west] {$\Y$};
\end{tikzpicture}
\caption{Retraction of a 9-cycle $\cX$ onto a $4$-point set $\cY$. The labels indicate the image of the corresponding vertex under the retraction $T$.}\label{fig:nonagon}
\end{figure}
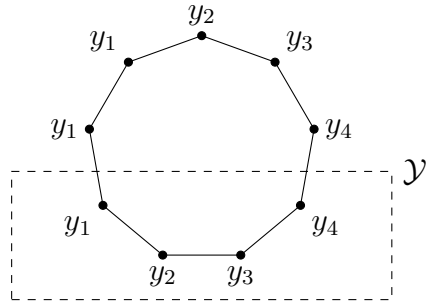

\begin{example}[Grid]\label{ex:grid}
Consider $\Z^d$ with the usual graph structure given by $Q_{xy}=1$ if $\abs{x-y}=1$ and $Q_{xy}=0$ otherwise. 
Let $\cY \subseteq \Z^d$ be a nonempty subset of the form $\cY = \cR \cap \Z^d$, where $\cR = \prod_{j=1}^d [a_j,b_j]$ is a hyperrectangle, and let $\cX$ be a connected subgraph of $\Z^d$ containing $\cY$.
We claim that $\cY$ has the retraction property. Indeed, it is readily checked that
a retraction from $\cX$ to $\cY$ can be obtained by mapping $x \in \cX$ to the point in $\cY$ that is closest to $x$ with respect to the Euclidean distance.
\end{example}

\begin{example}[2-point space]\label{ex:2pt}
  Assume that $Q$ takes values in $\{0, 1\}$ and let $x, y \in \cX$ with
  $Q_{xy} = 1$. A disjoint decomposition $\X=A_x\cup A_y$ with
  $x\in A_x$ and $y\in A_y$ is called an \emph{$x$-$y$ cut}.  An edge
  $(u,v) \in \cE$ is a \emph{cross} if $u \in A_x$ and
  $v\in A_y$. The subset $\{x,y\}$ has the retraction property if and
  only if there exists an $x$-$y$ cut such that no distinct crosses
  share a point. The correspondence between $x$-$y$ cuts with this
  property and retractions is given by $T^{-1}(x)=A_x$,
  $T^{-1}(y)=A_y$.
\begin{figure}[h!]
\begin{tikzpicture}[x=1cm,y=1cm]
\draw [fill=black] (0,0) circle (2.5 pt);
\draw (0,-0.1) node[anchor=north]{$x$};
\draw [fill=black] (2,0) circle (2.5 pt);
\draw (2,-0.1) node[anchor=north]{$y$};
\draw [fill=black] (3,1) circle (2.5 pt);
\draw [fill=black] (2,2) circle (2.5 pt);
\draw [fill=black] (0,2) circle (2.5 pt);
\draw [fill=black] (-1,1) circle (2.5 pt);
\draw (0,0)--(2,0);
\draw (2,0)--(3,1);
\draw (3,1)--(2,2);
\draw (2,2)--(0,2);
\draw (0,2)--(-1,1);
\draw (-1,1)--(0,0);
\draw[dashed] (0.5,-0.7)--(0.5,2.3);
\draw[dashed](0.5,2.3)--(-1.3,2.3);
\draw[dashed](-1.3,2.3)--(-1.3,-0.7);
\draw[dashed](-1.3,-0.7)--(0.5,-0.7);
\draw (-1.3,2.3) node[anchor=east]{$A_x$};
\draw[dashed] (1.5,-0.7)--(1.5,2.3);
\draw[dashed](1.5,2.3)--(3.3,2.3);
\draw[dashed](3.3,2.3)--(3.3,-0.7);
\draw[dashed](3.3,-0.7)--(1.5,-0.7);
\draw (3.3,2.3) node[anchor=west]{$A_y$};
\end{tikzpicture}

\caption{An $x$-$y$-cut associated with a retraction}
\end{figure}
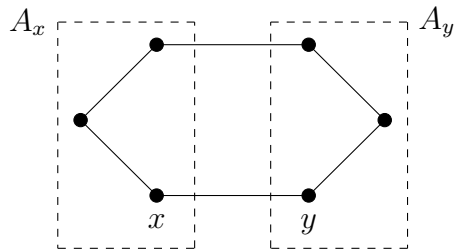
\end{example}

\begin{example}[Honeycomb lattice]\label{ex:honey}
Let $(\X,\mathcal{E})$ be a connected subgraph of the honeycomb lattice and define transition rates by setting $Q_{xy} = 1$ if $(x,y)\in\mathcal{E}$ and zero otherwise. 
Then each fundamental cell $\Y=\{y_1,\dots,y_6\}$ (see Figure \ref{fig:honeycomb}) has the retraction property. Indeed, to obtain a  retraction of $\cX$ onto $\cY$, we partition the plane into $6$ sectors separated by rays that originate at the centre of $\cY$ and intersect the midpoints of the sides of $\cY$ orthogonally. A retraction is then obtained by mapping each $x \in \cX$ to the unique $y \in \cY$ that belongs to the same sector (cf. Figure \ref{fig:honeycomb}).
\end{example}

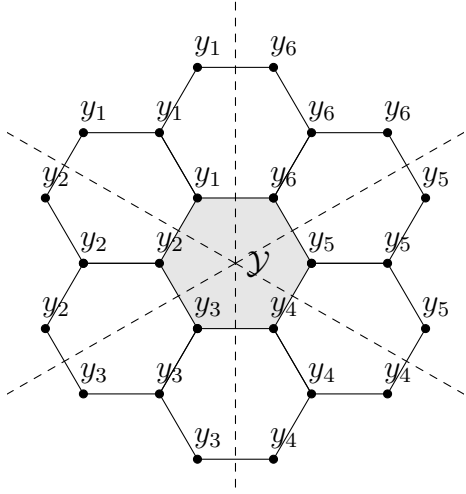
\begin{figure}[h!]
\begin{tikzpicture}[line cap=round,line join=round,>=triangle 45,x=1.0cm,y=1.0cm]
\clip(0,-3) rectangle (6,3.5);
\draw (1.,0.) -- (2.,0.) -- (2.5,0.8660254037844387) -- (2.,1.7320508075688776) -- (1.,1.7320508075688779) -- (0.5,0.8660254037844395) -- cycle;
\draw (2.,1.7320508075688776) -- (2.5,0.8660254037844387) -- (3.5,0.8660254037844386) -- (4.,1.7320508075688774) -- (3.5,2.5980762113533165) -- (2.5,2.598076211353317) -- cycle;
\fill[color=black,fill=black,fill opacity=0.1] (2.5,0.8660254037844387) -- (2.,0.) -- (2.5,-0.8660254037844386) -- (3.5,-0.866025403784439) -- (4.,0.) -- (3.5,0.8660254037844384) -- cycle;
\draw (4.,1.7320508075688774) -- (3.5,0.8660254037844386) -- (4.,0.) -- (5.,0.) -- (5.5,0.8660254037844374) -- (5.,1.7320508075688767) -- cycle;
\draw (4.,0.) -- (3.5,-0.866025403784439) -- (4.,-1.7320508075688776) -- (5.,-1.732050807568878) -- (5.5,-0.8660254037844397) -- (5.,0.) -- cycle;
\draw (4.,-1.7320508075688776) -- (3.5,-0.866025403784439) -- (2.5,-0.8660254037844393) -- (2.,-1.7320508075688779) -- (2.5,-2.5980762113533165) -- (3.5,-2.598076211353317) -- cycle;
\draw (2.5,-0.8660254037844386) -- (2.,0.) -- (1.,0.) -- (0.5,-0.866025403784438) -- (1.,-1.7320508075688763) -- (2.,-1.7320508075688772) -- cycle;
\draw[fill=black] (1.,0.) circle (1.5pt);
\draw(1.14,0.28) node {$y_2$};
\draw[fill=black](2.,0.) circle (1.5pt);
\draw (2.14,0.28) node {$y_2$};
\draw[fill=black] (2.5,0.8660254037844387) circle (1.5pt);
\draw (2.64,1.14) node {$y_1$};
\draw[fill=black](2.,1.7320508075688776) circle (1.5pt);
\draw (2.14,2.02) node {$y_1$};
\draw[fill=black](1.,1.7320508075688779) circle (1.5pt);
\draw (1.14,2.02) node {$y_1$};
\draw[fill=black](0.5,0.8660254037844395) circle (1.5pt);
\draw(0.64,1.14) node {$y_2$};
\draw[fill=black](3.5,0.8660254037844386) circle (1.5pt);
\draw (3.64,1.14) node {$y_6$};
\draw[fill=black](4.,1.7320508075688774) circle (1.5pt);
\draw(4.14,2.02) node {$y_6$};
\draw[fill=black](3.5,2.5980762113533165) circle (1.5pt);
\draw (3.64,2.88) node {$y_6$};
\draw[fill=black](2.5,2.598076211353317) circle (1.5pt);
\draw(2.64,2.88) node {$y_1$};
\draw(3.3,0) node {$\Y$};
\draw[fill=black](2.5,-0.8660254037844386) circle (1.5pt);
\draw[fill=black](3.5,-0.866025403784439) circle (1.5pt);
\draw(3.64,-0.58) node {$y_4$};
\draw[fill=black](4.,0.) circle (1.5pt);
\draw(4.14,0.28) node {$y_5$};
\draw[fill=black](3.5,0.8660254037844384) circle (1.5pt);
\draw[fill=black](4.,0.) circle (1.5pt);
\draw[fill=black](5.,0.) circle (1.5pt);
\draw(5.14,0.28) node {$y_5$};
\draw[fill=black](5.5,0.8660254037844374) circle (1.5pt);
\draw(5.64,1.14) node {$y_5$};
\draw[fill=black](5.,1.7320508075688767) circle (1.5pt);
\draw(5.14,2.02) node {$y_6$};
\draw[fill=black](4.,-1.7320508075688776) circle (1.5pt);
\draw(4.14,-1.46) node {$y_4$};
\draw[fill=black](5.,-1.732050807568878) circle (1.5pt);
\draw(5.14,-1.46) node {$y_4$};
\draw[fill=black](5.5,-0.8660254037844397) circle (1.5pt);
\draw(5.64,-0.58) node {$y_5$};
\draw[fill=black](2.5,-0.8660254037844393) circle (1.5pt);
\draw(2.64,-0.58) node {$y_3$};
\draw[fill=black](2.,-1.7320508075688779) circle (1.5pt);
\draw(2.14,-1.46) node {$y_3$};
\draw[fill=black](2.5,-2.5980762113533165) circle (1.5pt);
\draw(2.64,-2.32) node {$y_3$};
\draw[fill=black](3.5,-2.598076211353317) circle (1.5pt);
\draw(3.64,-2.32) node {$y_4$};
\draw[fill=black](1.,0.) circle (1.5pt);
\draw[fill=black](0.5,-0.866025403784438) circle (1.5pt);
\draw (0.64,-0.58) node {$y_2$};
\draw[fill=black](1.,-1.7320508075688763) circle (1.5pt);
\draw(1.14,-1.46) node {$y_3$};
\draw[dashed](3,0) -- +(30:10);
\draw[dashed](3,0) -- + (210:10);
\draw[dashed](3,0) -- + (150:10);
\draw[dashed](3,0) -- +(330:10);
\draw[dashed](3,10)--(3,-10);
\end{tikzpicture}
\caption{Part of the honeycomb lattice with a fundamental cell $\cY$. The  labels indicate the image of the corresponding vertex under the retraction.}\label{fig:honeycomb}
\end{figure}

\begin{example}[Trees]\label{ex:tree}
 Assume that the graph $(\cX, \cE)$ is a tree, i.e., it does not contain a cycle. Every subtree $\cY$ of $\cX$ has the retraction property, and a retraction can be constructed as follows: Fix a vertex $y \in \cY$. Since $\cX$ is a tree, for every $x\in \cX$ there is a unique path $\gamma$ without self-intersections connecting $x$ and $y$. The map assigning to $x$ the first point where the path $\gamma$ meets $\cY$ is a retraction of $\cX$ onto $\cY$. Note that the retraction property depends only on the graph $(\cX, \cE)$ and not on the choice of the transition rates $Q$ (as long as they give rise to the same graph).

\end{example}

\begin{theorem}(Extension of Hamilton--Jacobi subsolutions)
\label{thm:HJ-extension}
Let $(\cX, Q, \pi)$ be a Markov triple, and let $\Y$ be a connected subset of $\X$. If $\cY$ has the retraction property, then every Hamilton--Jacobi subsolution on $\cY$ can be extended to a Hamilton--Jacobi subsolution on $\cX$. 
\end{theorem}

\begin{proof}
Let $\phi$ be a Hamilton--Jacobi subsolution on $\cY$, and let $T$ be a retraction of $\X$ onto $\Y$.
Define $\bar \phi : \cX \to \R$ by $\bar \phi := \phi \circ T$, so that $\bar \phi|_\cY = \phi$ by (R1).
We will show that for any $\bar \nu \in \PX$, there exists $\nu \in \PY$ such that
\begin{equation}\label{eq:Hamiltonian_decr}
	 \ip{\dot{\bar\phi}_t,\bar\nu} 
 	+ \frac12 \|\nabla\bar\phi_t\|^2_{\bar\nu}
	\leq 
	 \ip{\dot\phi_t,\nu} 
 	+ \frac12 \|\nabla\phi_t\|^2_{\nu}
\end{equation}
for a.e. $t$.
To improve readability, we omit the subscript $t$. As $\phi \in \HJ_\cY$, the right-hand side of (\ref{eq:Hamiltonian_decr}) is nonpositive, so this suffices to prove the theorem.

For $\bar\nu\in\PX$ define $\nu \in \PY$ by $\nu := T_\# \bar \nu$. Clearly, 
\begin{align*}
	\ip{\dot{\bar{\phi}}, \bar \nu} 
	  = \ip{ {\dot\phi} \circ T , \bar \nu} 
	  = \ip{\dot{\phi}, T_\# \bar \nu} 
	  = \ip{\dot{\phi}, \nu} \;.
\end{align*}
It thus remains to show that $\|\nabla\bar\phi_t\|_{\bar\nu} \leq \|\nabla\phi_t\|_{\nu}$.

Splitting the sum we obtain
\begin{align*}
\|\nabla\bar\phi\|_{\bar\nu}^2 
&= \frac12  \sum_{x,x'\in \X}\Lambda(\bar\nu_x Q_{xx'}, \bar\nu_{x'} Q_{x'x})(\bar\phi_x - \bar\phi_{x'})^2
\\	&= \frac12 \sum_{\substack{y,y'\in \Y\\y\neq y'}}(\phi_y-\phi_{y'})^2\sum_{\substack{x\in T^{-1}(y)\\ x'\in T^{-1}(y')}}\Lambda(\bar\nu_x Q_{xx'}, \bar\nu_{x'} Q_{x'x} ) \;.
\end{align*}
The concavity and homogeneity of $\Lambda$ imply 
\begin{align*}
\begin{split}
&\sum_{\substack{x\in T^{-1}(y)\\ x'\in T^{-1}(y')}}\Lambda(\bar\nu_x Q_{xx'}, \bar\nu_{x'} Q_{x'x} )  
\\& \qquad \leq \Lambda\Bigg(\sum_{x\in T^{-1}(y)}\sum_{\substack{x'\in T^{-1}(y')}} \bar\nu_x Q_{xx'},
	\sum_{x'\in T^{-1}(y')}
\sum_{\substack{x\in T^{-1}(y)}} \bar\nu_{x'} Q_{x'x}\Bigg) \;.
\end{split}
\end{align*}
Given $x \in \cX$ and $y,y' \in \cY$ with  $y\neq y'$ and $T(x) = y$, the retraction property (R2) implies that $\sum_{x'\in T^{-1}(y')} Q_{xx'} \leq Q_{yy'}$ (and the same holds with primed and unprimed variables interchanged). 
Hence the monotonicity of $\Lambda$ yields
\begin{align*}
\Lambda\Bigg(&\sum_{x\in T^{-1}(y)}\sum_{\substack{x'\in T^{-1}(y')}} \bar\nu_x Q_{xx'},
	\sum_{x'\in T^{-1}(y')}
\sum_{\substack{x\in T^{-1}(y)}} \bar\nu_{x'} Q_{x'x}\Bigg) \\&\leq \Lambda\Bigg(Q_{yy'}\sum_{x\in T^{-1}(y)}\bar\nu_x, \ Q_{y'y} \sum_{x'\in T^{-1}(y')}\bar\nu_{x'}\Bigg)
=\Lambda(\nu_y Q_{yy'}, \nu_{y'} Q_{y'y} ) \;.
\end{align*}
Combining these inequalities, we infer that
\begin{align*}
\|\nabla\bar\phi\|_{\bar\nu}^2 
\leq \frac12\sum_{y,y'\in\Y}(\phi_y - \phi_{y'})^2\Lambda(\nu_y Q_{yy'}, \nu_{y'} Q_{y'y}) = 
\|\nabla\phi\|_{\nu}^2 \;,
\end{align*}
which completes the proof.
\end{proof}

The following result shows that any pair of measures supported in a
set $\cY$ with the retraction property, can be connected by a geodesic
supported in $\cY$.

\begin{theorem}[Weak locality under the retraction property]\label{thm:geodesics-extension}
Let $\Y$ be a subset of $\X$ with the retraction property. 
For all $\mu^0,\mu^1 \in \cP(\cX)$ with support in $\Y$ there exists a minimising $\W$-geodesic $(\mu_t)_{t \in [0,1]} \subseteq \cP(\cX)$ connecting $\mu^0$ and $\mu^1$ such that $\mu_t$ has support in $\Y$ for all $t\in[0,1]$.
\end{theorem}

\begin{proof} 
Let $(\mu_t)_t$ be a minimising geodesic in $\cP(\cY)$ satisfying the continuity equation \eqref{eq:cont-eq} with momentum vector field $(V_t)_t$. Consider the extension to $\X$ defined by $\bar\mu_t(x)=0$ if $x\notin \Y$ and $\bar V_t(x,x')=0$ if $x\notin \Y$ or $x'\notin \Y$. Clearly, $(\bar\mu_t,\bar V_t)_t$ has the same action as $(\mu_t,V_t)_t$.

Let $\eps > 0$. Since $(\mu_t,V_t)_t$ is a geodesic in $\PY$, Theorem \ref{thm:dual-W} (applied in $\cP(\cY)$) implies that there exists a Hamilton--Jacobi subsolution $\phi \in \HJ_\cY$ such that 
\begin{align*}
\ip{\phi_1,\mu^1} -\ip{\phi_0,\mu^0}
  \geq \frac12 \int_0^1 \cA(\mu_t,V_t) \dd t - \eps \;.
\end{align*}
By Theorem \ref{thm:HJ-extension}, $\phi$ can be extended to a Hamilton--Jacobi subsolution $\bar \phi \in \HJ_\cX$. In particular, using Theorem \ref{thm:dual-W} once more (this time in $\cP(\cX)$), 
\begin{align*}
\ip{\phi_1,\mu^1} -\ip{\phi_0,\mu^0}
= \ip{\bar\phi_1,\mu^1} -\ip{\bar\phi_0,\mu^0}
\leq  \frac12\cW_\cX^2(\mu^0,\mu^1) \;.
\end{align*}
Since $\eps > 0$ is arbitrary, it follows that $\int_0^1\A(\mu_t,V_t) \dd t \leq \cW_\cX^2(\mu^0,\mu^1)$, which yields the result.
\end{proof}

\section{Optimal transport avoids dead ends}
\label{sec:ends}
In this section we prove the intuitively natural statement that optimal curves do not transport mass into ``dead ends''. 
We formalise this concept by considering the gluing of two Markov triples along a vertex.

\begin{definition}[Gluing of Markov triples]\label{def:gluing}
Let $(\X_1,Q_1,\pi_1)$ and $(\X_2,Q_2,\pi_2)$ be Markov triples, and fix $x_1\in \X_1$, $x_2\in \X_2$. 
The \emph{gluing} of the two triples at $x_1, x_2$ is the Markov triple $(\cX,Q,\pi)$ defined by setting 
\begin{align*}
\X = (\X_1\sqcup \X_2) / \{x_1, x_2\}
\end{align*}
and $\ast = [x_1] = [x_2]$. 
For brevity, let us write $\cX_i' := \X_i \setminus \{x_i\}$.
We have canonical injections $\cX_1' \to \X$, $\cX_2' \to \X$, and we
identify elements of $(\X_1\sqcup \X_2) \setminus \{x_1,x_2\}$ with
their respective images.
We define transition rates $Q : \cX \times \cX \to \R$ by 
\begin{align*}
Q(x,y) = \left\{ \begin{array}{ll}
Q_i(x,y)
 & \text{if $x, y \in \X_i'$ }\;,\\
Q_i(x,x_i)
 & \text{if $x \in \cX_i'$ and $y = \ast$ },\\
Q_i(x_i,y)
 & \text{if $x = \ast$ and $y \in \cX_i'$ },\\
0
 & \text{otherwise}.\end{array} \right.
 \end{align*}
\end{definition}

It is easy to see that $Q$ is irreducible and reversible, and the unique
invariant probability measure is given by
\begin{align*}
\pi(x) =  \frac{1}{1 - \pi_1(\cX_1') \pi_2(\cX_2')} \times
\left\{ \begin{array}{ll}
\pi_1(x)\pi_2(x_2)
 & \text{if $x \in \cX_1'$}\;,\\
\pi_1(x_1)\pi_2(x)
 & \text{if $x \in \cX_2'$}\;,\\
\pi_1(x_1)\pi_2(x_2)
 & \text{if $x = \ast$}\;.\end{array} \right.
\end{align*}
\begin{definition}[Dead end]\label{def:dead-end}
Let $(\cX, Q, \pi)$ be a Markov triple, and let $\cX_1, \cX_2 \subseteq \cX$. We say that $\cX_2$ is a \emph{dead end} for $\cX_1$ (and vice versa) if the intersection of $\cX_1$ and $\cX_2$ contains exactly one point (denoted ``$\ast$''), and moreover, $Q(x,y) = Q(y,x) = 0$ whenever $x \in \cX_1'$ and $y \in \cX_2'$. Here, we write $\cX_i' = \cX_i \setminus \{ \ast \}$.
\end{definition}

\begin{remark}\label{rem:compatibility}
The notions of dead end and gluing of Markov triples are compatible in the following sense: Let $(\cX, Q, \pi)$ be a Markov triple, and suppose that $\cX_2 \subseteq \cX$ is a dead end for $\cX_1 \subseteq \cX$ with intersection point $\ast$. Then one recovers $(\cX, Q, \pi)$ by gluing together the restrictions of $\cX$ to $\cX_1$ and $\cX_2$ at $\ast$.
\end{remark}

\begin{proposition}\label{prop:gluing-extension}
Let $(\cX_1,Q_1,\pi_1)$ and $(\cX_2,Q_2,\pi_2)$ be Markov triples, and let $(\cX, Q, \pi)$ be the Markov triple obtained by gluing the triples at $x_1 \in \cX_1$ and $x_2 \in \cX_2$. Then $\X_1$ and $\X_2$ have the retraction property as subsets of $\X$. \end{proposition}

\begin{proof}
Define $T : \cX \to \cX_1$ by  $T(x) = x$ for $x \in \cX_1$ and $T(x) = \ast$ for $x \in \cX_2'$. One verifies that $T$ indeed defines a retraction by distinguishing cases.
\end{proof}

In view of Theorem \ref{thm:geodesics-extension}, the previous result implies that any two measures $\mu_0, \mu_1$ supported in (the image of) $\X_1$ can be connected by a geodesic that is supported in $\X_1$ for all times; i.e., weak locality holds. We will now show that in fact strong locality holds: \emph{any} geodesic connecting $\mu_0$ and $\mu_1$ has to be supported in $\X_1$.

\begin{theorem}\label{thm:no-dead-ends}
Let $(\cX_1,Q_1,\pi_1)$ and $(\cX_2,Q_2,\pi_2)$ be Markov triples, and let $(\cX, Q, \pi)$ be the Markov triple obtained by gluing the triples at $x_1 \in \cX_1$ and $x_2 \in \cX_2$. 
  If $(\mu_t)_{t\in[0,1]}$ is a geodesic in $(\P(\X),\W)$ with
  $\supp \mu_0, \supp \mu_1\subseteq \X_1$, then
  $\supp \mu_t\subseteq \X_1$ for all $t\in[0,1]$.
\end{theorem}

\begin{proof}
Let $t \mapsto V_t \in \R^{\cX \times \cX}$ be an anti-symmetric momentum vector field such that $(\mu, V)$ is a solution to the continuity equation with $\int_0^1 \cA(\mu_t, V_t) \dd t = \cW^2(\mu_0, \mu_1)$.
We define a new curve $t \mapsto \bar \mu_t \in \cP(\cX)$ by
\begin{align*}
\bar\mu_t(x)=\begin{cases}
\mu_t(x)& \text{if } x\in \X_1' \;,\\
\mu_t(\ast)+\sum_{y\in \X_2'}\mu_t(y)& \text{if } x=\ast\;,\\
0& \text{otherwise},
\end{cases}
\end{align*}
and a new anti-symmetric momentum vector field $t \mapsto \bar V_t \in \R^{\cX \times \cX}$ by
\begin{align*}
\bar V_t(x,y)=\begin{cases}
V_t(x,y)& \text{if } x,y\in \X_1' \cup \{\ast\}\;,\\
0& \text{otherwise}.
\end{cases}
\end{align*}
We claim that $(\bar \mu, \bar V)$ solves the continuity equation \eqref{eq:cont-eq} as well. 

Indeed, this statement trivially holds for any $x \in \cX \setminus \{ \ast \}$. To prove the claim at $\ast$, we note that for any $y \in \cX_2'$, 
\begin{align*}
	\ddt\mu_t(y) = \sum_{x \in \cX_2' \cup \{\ast\}} V_t(x,y)\;.
\end{align*}
Therefore, using the anti-symmetry of $V_t$,
\begin{align*}
	\sum_{y \in \cX_2'} \ddt\mu_t(y) 
		= \sum_{y \in \cX_2'}  \sum_{x \in \cX_2' \cup \{\ast\}} V_t(x,y)
		= \sum_{y \in \cX_2'}  V_t(\ast,y) \;.
\end{align*}
Furthermore, 
\begin{align*}
		\ddt\mu_t(\ast) = \sum_{y\in \cX_1'}V_t(y,\ast)
						+ \sum_{y\in \cX_2'}V_t(y,\ast) \;,
\end{align*}
hence by another application of the anti-symmetry,
\begin{align*}
	\ddt\bar\mu_t(\ast) 
		&= \ddt \mu_t(\ast) + \sum_{y\in \cX_2'} \ddt\mu_t(y)
		 =  \sum_{y\in \cX_1'} V_t(y,\ast)
		 =  \sum_{y\in \cX} \bar V_t(y,\ast) \;,
\end{align*}
which proves the claim.

For all $t \in (0,1)$ and $x, y \in \cX$, we clearly have
\begin{align*}
 A(\mu_t(x) Q(x,y), \mu_t(y) Q(y,x), V_t(x,y))
\geq A(\bar\mu_t(x) Q(x,y), \bar\mu_t(y) Q(y,x),\bar V_t(x,y)) \;.
\end{align*}
Moreover, if $\mu_t(\cX_2') > 0$ for some $t \in (0,1)$, then there exists $z \in \cX_2'$ such that $V_t(\ast, z) > 0$ and $\Lambda(\mu_t(\ast) Q(\ast,z), \mu_t(z) Q(z,\ast)) > 0$ for all $t$ on a set of positive measure in $(0,1)$. Therefore, 
\begin{align*}
	 A(\mu_t(\ast) Q(\ast,z), \mu_t(z) Q(z,\ast), V_t(\ast,z))
	> 0 = A(\bar\mu_t(\ast) Q(\ast,z), \bar\mu_t(z) Q(z,\ast),\bar V_t(\ast,z)) \;.
\end{align*}
This strict inequality contradicts the fact that $(\mu_t)_{t\in[0,1]}$ is a geodesic.
\end{proof}

\section{Nonlocality of optimal transport on the triangle}
\label{sec:triangle}

Consider a Markov triple $(\cX, Q, \pi)$ and a connected subset $\cY \subseteq \cX$. 
In this section we show that locality of geodesics in $\cP(\cY)$ may fail if $\cY$ does not have the retraction property.
We consider the simplest possible setting, where $(\cX, Q, \pi)$ corresponds to simple random walk on a triangle, and $\cY \subseteq \cX$ is a two-point set. 
We show that the canonical lift of a geodesic between Dirac measures on the two-point space is \emph{not} an optimal curve in $\cP(\cX)$, by constructing a competitor that transports mass along all edges. 

Throughout this section we make the following additional assumption on the mean $\Lambda$. 

\begin{assumption} \label{ass:bdy}
For any $s>0$ we have  
\begin{align}\label{eq:mean-limit}
\Lambda(s,t)\to\infty \quad \text{as } t\to\infty\;.
\end{align}
If $\Lambda(0,t)>0$ for $t>0$, then \eqref{eq:mean-limit} also holds for $s=0$.
\end{assumption}

Clearly, this assumption is satisfied for the arithmetic, geometric, and logarithmic means, but not for the harmonic mean.

The main result of this section relies on the following lemma concerning the variation of the action functional on cycles of arbitrary length.

\begin{lemma}\label{lem:derivA}
For $n \geq 3$, let $\cX = \mathbb{Z}/n\mathbb{Z}$ be equipped with transition rates $Q_{ij}$ such that $Q_{i,i+1},Q_{i+1,i}>0$ for all $i\in \X$ and $Q_{ij}=0$ otherwise.
Let $\mu, \nu \in \PX$, and let $V,U \in \R^{\X\times\X}$ be anti-symmetric, and such that both $\A(\mu,V)$ and $\A(\nu,U)$ are finite. 
Assume that $\mu_1,\mu_2>0$ and $\mu_i = 0$ for all $i\neq 1,2$, $V_{12} \neq 0$, and $V_{ij} = 0$ for all $\{i,j\} \neq \{1,2\}$ and that $U_{12}=0$. For $\alpha\in[0,1]$ we define $\mu^\a=(1-\a)\mu+\a\nu$ and $V^\a=(1-\a)V+\a U$. 
Then we have:
  \begin{align}\label{eq:derivA}
    &\lim_{\a\to0}\frac1\a
      \Big(\A(\mu^\a,V^\a)-\A(\mu,V)\Big)
       \\ & = \nonumber
    \frac{-V_{12}^2}{\Lambda\big(\mu_1 Q_{12}, \mu_2 Q_{21}\big)}\bigg[1+\frac{\partial_1\Lambda\big(\mu_1Q_{12},\mu_2Q_{21}\big)\nu_1Q_{12}+\partial_2\Lambda\big(\mu_1Q_{12},\mu_2Q_{21}\big)\nu_2Q_{21}}{\Lambda\big(\mu_1 Q_{12}, \mu_2 Q_{21}\big)}\bigg]
    \\ & \nonumber
\qquad +\sum_{i=3}^{n-1}\frac{U_{i,i+1}^2}{\Lambda\big(\nu_iQ_{i,i+1},\nu_{i+1}Q_{i+1,i}\big)} \; .
  \end{align}
\end{lemma}

\begin{proof}
  First note that
  \begin{align*}
    \A(\mu,V) 
    	= \frac{V_{12}^2}{\Lambda\big(\mu_1 Q_{12},\mu_2 Q_{21} \big)}\;.
  \end{align*}
  Using the $1$-homogeneity of $\Lambda$ we observe that
  \begin{align*}
    \A(\mu^\a,V^\a)
    & =\frac{(1-\a)V_{12}^2}{\Lambda\big(\big(\mu_1+\frac{\a}{1-\a}\nu_1\big)Q_{12},\big(\mu_2+\frac{\a}{1-\a}\nu_2\big)Q_{21}\big)}
   + 
   \a\sum_{i=3}^{n-1}\frac{U_{i,i+1}^2}{\Lambda\big(\nu_i Q_{i,i+1},\nu_{i+1}Q_{i+1,i}\big)}
   \\& \qquad 
   +
   \frac{\a U_{23}^2}{\Lambda\big(\big(\frac{1-\a}{\a}\mu_2+\nu_2\big)Q_{23},\nu_3 Q_{32}\big)}
   +
   \frac{\a U_{n1}^2}{\Lambda\big(\nu_n Q_{n1},\big(\frac{1-\a}{\a}\mu_1+\nu_1\big)Q_{1n}\big)}
   \\& =: T_1+T_2+T_3+T_4\;.
  \end{align*}
  Since $\mu_1,\mu_2>0$, the first term in \eqref{eq:derivA} is
  well-defined   
  and easily seen to be the limit of
  $(T_1-\A(\mu,V))/\a$. Obviously, $T_2/\a$ converges to the second
  term in \eqref{eq:derivA}. Finally, $T_3$ vanishes unless
  $U_{23}\neq0$. But in this case, 
 since $(1-\a)/\a\to\infty$
  as $\a\to0$, we see that $T_3/\a$ converges to zero as $\a\to0$ by Assumption \ref{ass:bdy}.  A
  similar argument applies to $T_4$.
\end{proof}

Now we can prove the nonlocality result.

\begin{theorem}\label{thm:triangle}
  Let $(\X,Q,\pi)$ be a Markov triple with $\X=\{1,2,3\}$ and such
  that $Q(x,y) > 0$ for all $x\neq y$. Let $(\mu_t)_{t \in [0,1]}$ be a
  $\W$-geodesic connecting $\mu_0 = \delta_1$ to $\mu_1 = \delta_2$. Then, $\mu_t(3)>0$ for some $0<t<1$.
\end{theorem}

As $\mu_t(3)>0$ for some $0<t<1$, the result implies that mass is transported along the edges $(1,3)$ and $(3,2)$.

\begin{proof}
  Suppose that the geodesic $(\mu_t,V_t)_{t\in[0,1]}$ transports only
  along the edge $(1,2)$, i.e., $V_t(2,3) = V_t(3,1) = 0$ for a.e.~$t \in (0, 1)$. Then $(\mu_t,V_t)$ must be given by the corresponding geodesic on the two point space $\{1,2\}$. Obviously, we have $\mu_t(1),\mu_t(2)>0$ for all $t\in(0,1)$. 
  Let $(\nu,U) \in \CE_1(\delta_1,\delta_2)$ be a curve of finite action
  such that $U_t(1,2)=0$ for a.e.~$t$ and
  $\nu_t(1),\nu_t(2),\nu_t(3)>0$ for all $0<t<1$. Define
  $(\mu^\a,V^\a)\in\CE_1(\delta_1,\delta_2)$ by
  $\mu^\a=(1-\a)\mu+\a\nu$ and $V^\a=(1-\a)V+\a U$ for $\a\in[0,1]$.
  Then Lemma \ref{lem:derivA} yields for a.e.~$t$:
  \begin{align*}
    \lim_{\a\to 0}\frac{1}{\a}\big[\A(\mu_t^\a,V_t^\a)-\A(\mu_t,V_t)\big]<0\;.
  \end{align*}
  Consequently, there exists $\a>0$ such that 
  \begin{align*}
    \int_0^1\A(\mu_t^\a,V_t^\a)\dd t < \int_0^1\A(\mu_t,V_t)\dd t\;,
  \end{align*}
  contradicting the optimality of $(\mu,V)$.
\end{proof}

\bibliography{bib_geodesics}{}
\bibliographystyle{alpha}
\end{document}